\documentclass[12pt,reqno]{amsart}

\numberwithin{equation}{section}

\usepackage{enumerate}
\usepackage{amssymb}
\usepackage{amsmath}
\usepackage{amscd}
\usepackage{amsthm}
\usepackage{amsfonts}
\usepackage{graphicx}
\usepackage[all]{xy}
\usepackage{verbatim}
\usepackage{hyperref}

\newtheorem{theorem}{Theorem}[section]

\newtheorem{lemma}[theorem]{Lemma}
\newtheorem{corollary}[theorem]{Corollary}
\theoremstyle{definition}\newtheorem{definition}[theorem]{Definition}

\newtheorem{conjecture}[theorem]{Conjecture}

\theoremstyle{definition}

\theoremstyle{definition}
\theoremstyle{definition}\newtheorem{remark}[theorem]{Remark}
\theoremstyle{definition}\newtheorem*{acknowledgments}{Acknowledgments}

\newcommand{\al}{\alpha}
\newcommand{\be}{\beta}
\newcommand{\ga}{\gamma}
\newcommand{\Ga}{\Gamma}
\newcommand{\del}{\delta}

\newcommand{\lam}{\lambda}
\newcommand{\Lam}{\Lambda}
\newcommand{\eps}{\epsilon}

\newcommand{\sig}{\sigma}

\newcommand{\om}{\omega}
\newcommand{\Om}{\Omega}
\newcommand{\vphi}{\varphi}

\newcommand{\cO}{\mathcal{O}}

\newcommand{\bC}{\mathbb{C}}

\newcommand{\bR}{\mathbb{R}}
\newcommand{\bZ}{\mathbb{Z}}
\newcommand{\bQ}{\mathbb{Q}}

\newcommand{\bT}{\mathbb{T}}

\newcommand{\goa}{\mathfrak{a}}

\newcommand{\gog}{\mathfrak{g}}
\newcommand{\goh}{\mathfrak{h}}
\newcommand{\gol}{\mathfrak{l}}

\newcommand{\gou}{\mathfrak{u}}

\newcommand{\SL}{\operatorname{SL}}

\newcommand{\GL}{\operatorname{GL}}

\newcommand{\re}{\operatorname{Re}}
\newcommand{\im}{\operatorname{Im}}

\newcommand{\lie}{\operatorname{Lie}}
\newcommand\norm[1]{\left\|#1\right\|}
\newcommand\set[1]{\left\{#1\right\}}
\newcommand\pa[1]{\left(#1\right)}

\newcommand\idist[1]{\langle#1\rangle}
\newcommand\av[1]{\left|#1\right|}
\newcommand\on[1]{\operatorname{#1}}
\newcommand\diag[1]{\operatorname{diag}\left(#1\right)}

\newcommand{\onto}{\xymatrix{\ar@{>>}[r]&}}
\newcommand{\da}[4]{\xymatrix{#1 \ar@<.5ex>[r]^{#2} \ar@<-.5ex>[r]_{#3} & #4}}

\begin{document}
\title{Homogeneous orbit closures and applications}
\author{Elon Lindenstrauss and Uri Shapira}
\thanks{E.L. was partially supported by NSF grants DMS-0554345 and DMS-0800345 and ISF grant 983/09. U.S. was partially supported by ISF grants 983/09 and 1157/08}
\dedicatory{In memory of Dan Rudolph}
\begin{abstract}
We give new classes of examples of orbits of the diagonal group in the space of unit volume lattices in $\bR ^ d$ for $d \geq 3$ with nice (homogeneous) orbit closures, as well as examples of orbits with explicitly computable but irregular orbit closures. We give Diophantine applications to the former, for instance we show that for all $\ga,\del \in\bR$
$$\liminf_{\av{n}\to \infty}\av{n}\idist{n \sqrt[3]{2}-\ga}\idist{n \sqrt[3]{4}-\del}=0,$$
where $\idist{c}$ denotes the distance of a real number $c$ to the integers.
\end{abstract}

\maketitle
\section{Introduction and results}
Let $G$ be a Lie group and $\Ga<G$ be a closed subgroup. The space $X=G/\Ga$ is a \textit{homogeneous space} on which $G$ acts transitively by left multiplication. In homogeneous dynamics one studies the action of a closed subgroup,
$H<G$, on $X$. One of the basic questions one can ask is to analyze orbit closures, $\overline{Hx}$, for various points $x\in X$. We will shortly restrict our discussion to a specific example, having number theoretic applications in mind, but for the meantime, let us make the following definitions:
\begin{definition}\label{regular points}
\begin{enumerate}
  \item An $H$-orbit $Hx$ is \emph{periodic} if $Hx$ supports an $H$-invariant probability measure.
  \item An $H$-orbit $Hx$ is  \emph{$H$-regular} if $\overline{Hx}=Lx$ for some closed subgroup $H<L<G$.
\item An $H$-orbit is \emph{$H$-regular of periodic type} if furthermore $Lx$ is a periodic $L$-orbit.
\end{enumerate}
   A point $x$ is said to be $H$-periodic, $H$-regular or $H$-regular of periodic type if the corresponding $H$-orbit $Hx$ has this property.
\end{definition}

A simple situation where every point is $H$-regular is given by the action of a closed subgroup $H<\bR^d$ on the torus $\bT^d=\bR^d/\bZ^d$.
It is well known that in this situation any point $x\in \bT^d$ is $H$-regular of periodic type.
Moreover, the commutativity of $\bR^d$ implies
that the group $L$ which satisfies $Lx=\overline{Hx}$ does not depend on $x$.
A much deeper theorem ensuring such regularity is the following fundamental result of M. Ratner (see~\cite{R-Duke} Theorems A and B):
\begin{theorem}[Ratner's Orbit Closure Theorem]
Assume $\Ga<G$ is a lattice and $H<G$ a closed subgroup generated by one parameter
unipotent subgroups of $G$. Then any point $x\in G/\Ga$ is $H$-regular of periodic type.
\end{theorem}

Apart from their considerable intrinsic interest, the study of orbit closures for group actions on homogeneous spaces has numerous applications to other areas of mathematics, notably to number theory and the theory of Diophantine approximations. For example, in the mid 1980s, G.A. Margulis established a long-standing conjecture of Oppenheim regarding values of indefinite quadratic forms by analyzing orbit closures for the action of  the group preserving such an indefinite form on $\SL_3(\bR)/\SL_3(\bZ)$ (see~\cite{M-Oppenheim,Margulis-Fields-medal}).

In Margulis' proof of the Oppenheim Conjecture the acting group is generated by unipotent one parameter groups.
We shall discuss in this paper the opposite situation where the acting subgroup $H<G$ is diagonalizable. 
In fact we will confine our discussion to the specific setting of $$X_d=G/\Ga\qquad G=\SL_d(\bR)\qquad\Ga=\SL_d(\bZ)\qquad d \ge2.$$
We denote the image of $e \in G$ under the projection $G \to G/\Gamma$ by $e_\Gamma$. More generally, if $g\in G$, we write $ge_\Ga$ for the image of $g$ under this projection.
The space $X_d$ is identified in a natural way with the
space of unit volume lattices in $\bR^d$. Under this identification $ge_\Ga\in X_d$ corresponds to the lattice spanned by the
columns of the matrix $g$ (hence $e_\Ga$ corresponds to the standard lattice $\bZ^d$), and the action of $G$ on $G/\Ga$ coincides with the action of $G$ on the space of lattices induced from the action of $G$ on $\bR^d$.
Unless stated otherwise, we shall view elements of $\bR^d$ as column vectors. We let
\begin{equation}\label{the diagonal group}
A=\set{\diag{e^{t_1},\dots,e^{t_d}}:t_i\in\bR,\;\sum_1^d t_i=0},
\end{equation}
denote the subgroup of
$G$ consisting of diagonal matrices with positive diagonal entries (the group $A$ depends implicitly on $d$). 

\subsection {Regular and irregular $A$-orbits in $X_d$}
It is well known that when $d=2$
 there are many irregular points for the $A$-action (though by ergodicity of the $A$-action, almost every $x \in X_2$ has a dense orbit under $A$, hence in particular is $A$-regular).
 Indeed, in this case there are points $x\in X_2$ such that the Hausdorff dimension of the orbit closure $\overline{Ax}$ is not an integer, including points  with a bounded $A$-trajectory.

The situation is expected to change dramatically for $d\ge 3$. For example, for $d \geq 3$ we have the following conjecture essentially due to Cassels and Swinnerton-Dyer \cite {CaSD}, recast in dynamical terms by Margulis \cite{Margulis-Fields-medal}:
\begin{conjecture} For $d \geq 3$ every bounded $A$-orbit in $X _ d$ is periodic.
\end{conjecture}
While this conjecture remains open, Einsiedler, Katok and Lindenstrauss \cite{EKL} have shown that for $d \ge 3$, for any $x \in X _ d$ with a bounded $A$-orbit, the orbit closure $\overline {A x}$ has the same dimension as $A$.
In contrast to the unipotent case, it is easy to see that even for $d \ge 3$ there are points in $X _ d$ with an irregular $A$-orbit. For example, take any point in $ X_2$ whose orbit under the one parameter diagonal subgroup of $\SL _ 2 (\bR)$ is not $A$-regular, and let $\Lambda' $ denote the corresponding lattice in $ \bR^2$. Then the point in $X_3$  corresponding to the lattice $\Lambda=\Lambda' \oplus \bZ$
 has an irregular $A$-orbit.\label{simple construction} It seems reasonable to expect that there should be some countable union of explicit proper subvarieties $V_i \subset G$ so that every $x \not \in \bigcup_{i} V_i$ has a regular $A$-orbit (indeed, a dense $A$-orbit), but nailing down an explicit conjecture in this direction has proved to be somewhat tricky.
\medskip

The aim of this paper is to exhibit new explicit examples of $A$-regular points of periodic type as well as explicit examples of irregular points. We then use the results to obtain nontrivial information on Diophantine approximations
of algebraic numbers.
\medskip

The following theorem gives an explicit construction of interesting $A$-regular points of periodic type (see \S\ref{geometric embeddings} for definitions and terminology, e.g.\ of geometric embedding).

\begin{theorem}\label{theorem 1}
Let $K$ be a number field of degree $d\ge 3$ which is not a CM field\footnote{Recall that a number field $K$ is said to be a CM-field if it is a totally complex quadratic extension of a totally real field.
}, and let $\vphi:K\to\bR^d$ be a geometric embedding of $K$. Let $\Lam\subset K$ be
a lattice and  $x_\Lam\in X_d$ be the point corresponding to the lattice $\vphi(\Lam)$ in $\bR^d$ after normalizing its volume. Then $x_\Lam$ is $A$-regular of periodic type.
\end{theorem}

Theorem~\ref{theorem 1} is a special case of Theorem~\ref{theorem 2}, whose statement is deferred to the next section.
When $K$ is totally real (i.e.\ it has only real embeddings) the $A$-orbit of the point $x_\Lambda$ is periodic, hence trivially $A$-regular of periodic type. Barak Weiss and the first named author \cite{LW} have shown that any point $x \in X _ d$ for which $\overline {Ax} \ni x _ \Lambda$ (with $x_\Lambda$ arising as above from a totally real field $K$, and $d\geq3$) is $A$-regular of periodic type, and this theorem can also be used to construct non-obvious explicit $A$-regular points. Theorem~\ref{theorem 1} and the results of \cite{LW} implies that in fact if $\overline {Ax} \ni x _ \Lambda$ then $x$ is $A$-regular of periodic type whenever $K$ satisfies the conditions of Theorem~\ref{theorem 1} (cf.\ Corollary~\ref{generalization of LW}).

In the other direction, in \cite{Shap1} the second named author established that there exist irregular $A$-orbits in $X_3$ not of the form outlined above on p.~\pageref{simple construction}. This is somewhat surprising, as it contradicts an influential conjecture regarding the orbit closure of multidimensional diagonalizable group by Margulis~\cite[Conjecture 1.1]{M-problems} (Maucourant~\cite{Mau}  has already given a counterexample to this conjecture when instead of taking the full diagonal group $A$, one takes a suitable multidimensional subgroup; we have learned while finalizing this text that Tomanov has also constructed interesting counterexamples somewhat similar to the class considered here for a different group $G$). The proof given in~\cite{Shap1} was indirect. In \S~\ref{irregular} we further analyze these examples and give a full description of the orbit closures in these cases.
Keeping notational introduction to the minimum, we state here the following theorem. A more accurate version in the form of Theorem~\ref{theorem 3'} appears in \S~\ref{irregular}.
For a vector $v\in\bR^{d-1}$, we let
\begin{equation}\label{hor sgrps}
h_v=\pa{
\begin{array}{ll}
1&0\\
v & I_{d-1}
\end{array}
},\quad
g_v=\pa{
\begin{array}{ll}
1&v^t\\
0 & I_{d-1}
\end{array}
}
\end{equation}
where $I_{d-1}$ denotes the identity matrix of dimension $d-1$ and the $0$'s denote the corresponding trivial vectors.
Let $x_{v},z_v\in X_d$, denote the lattices spanned by the columns of $h_v$ and $g_v$ respectively.
\begin{theorem}\label{irregular points theorem}
Let $v=(\al,\be)^t\in\bR^2$ be such that $\al,\be$ are irrational and $1,\al,\be$ linearly dependent over $\bQ$. Then
there exist two  reductive groups $H^{(i)}, i=1,2$ (containing $A$), and two lattices $y_1,y_2\in X_3$, such that the orbits $H^{(i)}y_i$ are closed and such that
\begin{enumerate}
\item $\overline{Ax_v}\subset Ax_v\cup H^{(1)}y_1 \cup H^{(2)}y_2$,
\item\label{T002} the orbit $Ax_v$ is disjoint from $H^{(i)}y_i$,
\item\label{T003} $\overline{Ax_v}\cap H^{(i)}y_i\ne\emptyset$.
\end{enumerate}
A corresponding statement for the lattice $z_v$ holds (with different groups $H^{(i)}$).
 In particular, $x_v,z_v$ are irregular for the $A$-action.
\end{theorem}
\begin{remark}
In fact, it is not hard to see that for (Lebesgue) almost any $\al$, for any $\be$ as in Theorem~\ref{irregular points theorem}, one actually has the equality $$\overline{Ax_v}= Ax_v\cup H^{(1)}y_1 \cup H^{(2)}y_2.$$
\end{remark}

\subsection{Diophantine approximations of algebraic vectors}
One of the main motivations which led to the results appearing in this paper was to interpret dynamically the work of Cassels and Swinnerton-Dyer, who proved in~\cite{CaSD} that if $\al,\be\in\bR$ are two algebraic numbers belonging to the same cubic number field, then they satisfy the following conjecture of Littlewood:

\begin{conjecture}[Littlewood, c. 1930] For any pair of real numbers $\al,\be \in \bR$,
\begin{equation}\label{Littlewood}
\liminf_{\av{n}\to\infty}\av{n}\idist{n\al}\idist{n\be}=0.
\end{equation}
\end{conjecture}

In dynamical terms, the Cassels and Swinnerton-Dyer result amounts to  showing that for $v=(\al,\be)^t$, the orbit of $x_v$
under an appropriate open semigroup of $A$ is unbounded. We fully analyze the orbit closures in this and more general cases, and this stronger statement has further Diophantine implications.
In order to state our results on Diophantine approximations we give the following definition:
\begin{definition}\label{property C}
A vector $v\in\bR^d$ is said to have property C (after Cassels) of the first type, if the following statement holds:
\begin{align}\label{C2}
\textrm{For all }\vec{\ga}\in\bR^d \quad \liminf_{\av{n}\to\infty}\av{n}\prod_1^d\idist{nv_i-\ga_i}=0.
\end{align}
It is said to have property C of the second type, if the following statement holds:
\begin{align}\label{C1}
\textrm{For all }\ga\in\bR \quad \liminf_{\vec{n}\in\bZ^d, \prod\av{n_i}\to\infty}\pa{\prod_1^d\av{n_i}}\idist{\sum_1^d n_iv_i -\ga}=0.
\end{align}
\end{definition}
For $d=1$, it was shown by Khinchine in the early 1920's that numbers $v\in\bR$ with property C (the two notions of this property coincide in this case)
 do not exist (see~\cite{D}). The question of whether in higher dimensions vectors with property C exist was open until recently. In~\cite{Shap1}, the second named author proved that almost any vector in $\bR^d$ ($d\ge 2$) has property C of both types. Moreover, it was shown there that if $1,\al,\be$ form a basis for a totally real cubic number field, then the vector $(\al,\be)^t$ has property $C$ of both types.
We give the following more general result covering the case of non-totally real cubic fields and number field of higher degree:
\begin{theorem}\label{theorem 3}
Let $1,\al_2,\dots,\al_d\in \bR$ be a basis for a number field of degree $d\ge 3$ over $\bQ$. Then the vector $(\al_2,\dots,\al_d)^t\in\bR^{d-1}$ has property C of both types.
\end{theorem}
Note that for the vector $(\al,\be)^t$ to have property C of the first type is a much stronger property than for it to satisfy Littlewood's conjecture. Note also that when $\al,\be$ are linearly dependent over $\bQ$, then $\al,\be$ satisfy Littlewood's conjecture almost trivially, while the vector $(\al,\be)^t$ does not have property C of any type; see Theorem 1.3 in~\cite{Shap1}. In this respect, Theorem~\ref{theorem 3} is a strengthening of the aforementioned  result of Cassels and Swinnerton-Dyer.

\medskip
We shall use the following definition from \cite{Shap1}
\begin{definition}
A
lattice $x\in X_d$ is said to be GDP\footnote{GDP is an acronym for all Grids have Dense Products.}, if for any vector $w\in\bR^d$, the set of products $\set{\prod_1^d(u_i+w_i):u\in x}$ is dense in $\bR$.
\end{definition}
In~\cite{Shap1} it is shown that if the lattice $x_v$ (resp.\ $z_v$) is GDP, then $v$ has property $C$ of type 1 (resp.\ 2).  Hence, Theorem~\ref{theorem 3} follows from the next two theorems:
\begin{theorem}\label{theorem 4}
Let $d\ge 3$ and $x\in X_d$ be given. If $x$ is $A$-regular of periodic type, then either $x$ is $A$-periodic, or $x$ is GDP.
\end{theorem}
\begin{theorem}\label{theorem 5}
Let $v=(\al_2,\dots,\al_d)^t\in\bR^{d-1}$ be as in Theorem~\ref{theorem 3}. Then $x_v,z_v$ are $A$-regular of periodic type but not $A$-periodic.
\end{theorem}
Theorem~\ref{theorem 4} is proved at the end of \S\ref{main proofs} and Theorem~\ref{theorem 5} is proved in \S~\ref{applications}.

\section{Lattices coming from number fields}
In this section we study in some detail the lattices coming from number fields, which are the subject of Theorems~\ref{theorem 1} and~\ref{theorem 2}.
We begin by fixing some of the notation that will accompany us through this paper. Throughout this section we fix $d\ge 2$ and $r,s\ge 0$ to
be integers such that $d=r+2s$.
\subsection{Maximal tori in $G$}\label{maximal tori}
Given square matrices $B_1\dots B_n$ of any dimensions, we denote by $$\diag{B_1\dots B_n}$$ the block diagonal square matrix formed by the $B_i$'s.
For a complex number $\om$ we let
$$R_\om=\pa{\begin{array}{ll}
\re{\om}&-\im{\om}\\
\im{\om}&\re{\om}
\end{array}}.
$$
Let
\begin{equation}\label{rstorus}
T^{(r,s)}=\set{\diag{a_1,\dots,a_r,R_{\om_1},\dots,R_{\om_s}}\in G : a_i\in\bR_+, \om_i\in \bC}.
\end{equation}
$T^{(r,s)}$ is the connected component of the identity of a maximal torus in $G$ and any connected component of the identity in a maximal torus in $G$ is conjugate to exactly one of the $T^{(r,s)}$'s. When $s=0$ we denote $A=T^{(d,0)}$. The \textit{split part} of $T^{(r,s)}$ is defined to be
\begin{equation}\label{rssplit}
A_{r,s}=A\cap T^{(r,s)}.
\end{equation}
Let $G_\bC=\SL_d(\bC)$ and $A_\bC$ be the group of diagonal matrices in $G_\bC$.
Set $B=\frac{1}{\sqrt{2}}\pa{\begin{array}{ll}
1&i\\
i&1
\end{array}}$ and denote
\begin{equation}\label{theta}
\theta_{r,s}=\on{diag}\big(\underbrace{1,\dots,1}_{r},\underbrace{B,\dots,B}_{s}\big)\in G_\bC.
\end{equation}
Denote conjugation by $\theta_{r,s}$ in $G_\bC$ by $g\mapsto\tilde{g}$. For a subgroup $H<G_\bC$ we also denote
\begin{equation}\label{conju}
\tilde{H}=\theta_{r,s}H\theta_{r,s}^{-1}.
\end{equation}
The reader would easily verify that
$\tilde{T}^{(r,s)}\subset A_\bC$ and more precisely
\begin{equation}\label{tilde}
g=\diag{a_1,\dots,a_r,R_{\om_1},\dots,R_{\om_s}}\Rightarrow \tilde{g}=\diag{a_1,\dots ,a_r,\om_1,\bar{\om}_1,\dots,\om_s,\bar{\om}_s}.
\end{equation}
Note that $g\mapsto\tilde{g}$ is the identity map on $A_{r,s}$.\\
For $1\le i\ne j\le d$ let $\chi_{i,j}:A_\bC\to \bC^*$ be the character defined by
\begin{equation}\label{chi}
\chi_{ij}\pa{\diag{a_1,\dots,a_d}}=\frac{a_i}{a_j}.
\end{equation}

\subsection{Maximal parabolics}
For $1\le k\le d-1$ let $a_k(t)$ denote the one parameter subgroup of $A$ given by
\begin{equation}\label{a_k}
a_k(t)=\on{diag}\big(\underbrace{e^{(d-k)t},\dots,e^{(d-k)t}}_{k},\underbrace{e^{-kt},\dots,e^{-kt}}_{d-k}\big).
\end{equation}
To the one parameter group $a_k(t)$ we can attach two maximal parabolic subgroups of $G$, namely the weak-unstable and weak-stable horospherical subgroups of $a_k(1)$. More precisely, let
\begin{equation}\label{P_k}
P_k^+=\set{\pa{\begin{array}{ll}
B&C\\
0&D
\end{array}}\in G}\;\;;\;\;P_k^-=\set{\pa{\begin{array}{ll}
B&0\\
C&D
\end{array}}\in G},
\end{equation}
where in the above equations $B$ and $D$ are square matrices of dimensions $k,d-k$ respectively and $C$ and $0$ are rectangular matrices
of the obvious dimensions, $0$ denoting here the matrix all of whose entries equal zero.
\subsection{Geometric embeddings}\label{geometric embeddings}
Let $K$ be a number field of degree $d$ over $\bQ$. We say that $K$\textit{ is of type} $(r,s)$ if it has $r$ distinct real embeddings $\sig_i:K\to\bR,\;\, i=1\dots r$ and $s$ non-conjugate complex embeddings $\sig_i:K\to \bC,\;\,i=r+1\dots r+s$. A \textit{geometric embedding} of $K$ in $\bR^d$ is a map $\vphi:K\to\bR^d$ whose coordinates are the real embeddings and the real and
imaginary parts of the non-conjugate complex embeddings; i.e.\ up to a permutation of the coordinates it is the following map:
\begin{equation}\label{g.e}
\vphi=\big(\sig_1,\dots,\sig_r,\underbrace{\dots\re\sig_{r+i},\im\sig_{r+i}\dots}_{s}\big).
\end{equation}
We shall always work with geometric embeddings as in~\eqref{g.e} and will not allow any permutation for ease of notation.
Let $\al_1,\dots,\al_d$ be a basis of $K$ over $\bQ$. The $\bZ$-module $\Lam=\on{Span}_\bZ\set{\al_i}$ is called a $\textit{lattice}$ in $K$. It is well known that the geometric embedding of $\Lam$, $\vphi\pa{\Lam}\subset \bR^d$, is a lattice in $\bR^d$. Hence, by normalizing the covolume to be one,  $\Lam$ defines a point in $X_d$ which we denote by $x_\Lam$. We refer to such a lattice $x_\Lam$ as a \textit{lattice coming from a number field of
type} $(r,s)$.
We now have enough terminology to state one of the main results in this paper. Theorem~\ref{theorem 5} is a consequence of the following theorem which
generalizes Theorem~\ref{theorem 1} when $r>0$. It is proved together with Theorem~\ref{theorem 1} in the next section.
\begin{theorem}\label{theorem 2}
Let $x_\Lam\in X_d$ be a lattice coming from a number field of type $(r,s)$, let $k$ be a number co-prime to $d$ such that $\set{a_k(t)}<T^{(r,s)}$, and let $p\in P_k^+\cup P_k^-$. Then $px_\Lam$ is
$A$-regular of periodic type.
\end{theorem}
\subsection{A few lemmas}\label{few lemmas}
We now describe the connection between lattices $x_\Lam$ coming from number fields of type $(r,s)$, and the tori $T^{(r,s)}$. We shall shortly prove that the orbit $T^{(r,s)}x_\Lam$, is compact and homeomorphic to $\bT^{d-1}$. Moreover, we shall analyze to some extent the closure of the orbit $A_{r,s}x_\Lam$ in $T^{(r,s)}x_\Lam$.
We shall use hereafter the following notation: If a group $H$ acts on a set $X$ then for $x\in X$, $H_x$ denotes the stabilizer
of $x$ in $H$.
As $T^{(r,s)}$ is isomorphic as a group to  $\bR^{r+s-1}\times \bT^s$ we have the following basic lemma which
is left without proof.
\begin{lemma}\label{regularity of torus}
Let $x \in X_d$ be given. The orbit $T^{(r,s)}x$ is compact if and only if the stabilizer $T^{(r,s)}_{x}$ contains a free abelian group with $r + s - 1$ generators. Moreover, if $T^{(r,s)}x$ is compact then for any closed connected subgroup $H<T^{(r,s)}$, $\overline{Hx}=Lx$ for some closed connected subgroup $H<L<T^{(r,s)}$.
\end{lemma}

Let $K$ be a number field of type $(r,s)$ with geometric embedding $\vphi$ as above. Let $\psi:K\to \on{M}_d(\bR)$ be the map defined by (recall the notation of \S\S~\ref{maximal tori})
\begin{equation}\label{psi}
\psi(\alpha)=\diag{\sig_1(\al),\dots,\sig_r(\al),R_{\sig_{r+1}(\al)},\dots,R_{\sig_{r+s}(\al)}}.
\end{equation}
Observe that if we denote multiplication by $\al$ in $K$ by $m_\al$, then the following diagram commutes
\begin{equation}\label{commutative}
\xymatrix{K\ar[r]^\vphi\ar[d]_{m_\al}&\bR^d\ar[d]^{\psi(\al)}\\
K\ar[r]_\vphi&\bR^d.}
\end{equation}
The \textit{associated order} of a lattice $\Lam\subset K$ is defined to be
$$\cO_\Lam=\set{\al\in K:m_\al(\Lam)\subset \Lam}.$$
The reader would easily convince himself that $\cO_\Lam$ is a subring of $K$ and that the group of units of this ring is given by
$\cO^*_\Lam=\set{\al\in K :m_\al(\Lam)=\Lam}.$ We denote
\begin{equation}\label{positive units}
\cO_{\Lam,+}^*=\set{\al\in\cO_\Lam^*:\sig_i(\al)>0, i=1\dots r}.
\end{equation}
It follows from~\eqref{commutative},~\eqref{positive units} that $\cO_{\Lam,+}^*$ is embedded via $\psi$ in the stabilizer of $x_\Lam$ in $T^{(r,s)}$
(note that the determinant of $\psi(\al)$ is equal to 1 for any $\al\in\cO_{\Lam,+}^*$). In fact it is not hard to verify that this embedding is onto;
i.e.\
\begin{equation}\label{stab}
T^{(r,s)}_{x_\Lam}=\psi\pa{\cO_{\Lam,+}^*}.
\end{equation}
$\cO_{\Lam,+}^*$ is a subgroup of finite index in $\cO_\Lam^*$ and hence by Dirichlet's unit theorem
\begin{equation}\label{Dirichlet's}
\cO_{\Lam,+}^*\simeq \mu\times \bZ^{r+s-1},
\end{equation}
where $\mu$ is a finite group of roots of unity. Equations~\eqref{stab} and~\eqref{Dirichlet's} together with
Lemma~\ref{regularity of torus} imply the following
\begin{lemma}\label{compact orbits}
Let $x_\Lam\in X_d$ be a lattice coming from a number field of type $(r,s)$, then the orbit $T^{(r,s)}x_\Lam$ is compact.
\end{lemma}
In order to state the next lemma we introduce some more terminology: A subgroup $H<A_\bC$ is an \textit{equiblock diagonal group} if there are numbers
$d_1,d_2$ such that $d=d_1d_2$ and a partition
of the indices $\set{1\dots d}$ into subsets, $I_\ell$, $\ell=1\dots d_2$ of equal size $d_1$ such that for any $\ell$ and any $i\ne j\in I_\ell$ $H<\on{Ker}\pa{\chi_{ij}}.$ In that case $d_1$ is referred to as the \textit{size of the block}.
\begin{lemma}\label{equiblock}
Let $x_\Lam\in X_d$ be a lattice coming from a number field $K$ of type $(r,s)$. Let $H<T^{(r,s)}$ be a subgroup such that $Hx_\Lam$ is closed.
Suppose $\tilde{H}<\on{Ker}(\chi_{ij})$ (see~\eqref{conju} for notation) for some $i\ne j$. Then $\tilde{H}_{x_\Lam}$ is an equiblock diagonal group with block size $d_1>1$. Moreover, if $d=d_1d_2$ then there is a subfield $K'\subset K$ of degree $d_2$ over $\bQ$ such that $H_{x_\Lam}\subset \psi\pa{\cO_{\Lam,+}
\cap K'}.$
\end{lemma}
\begin{proof}
 From~\eqref{tilde},~\eqref{stab} it follows that
$$\tilde{T}^{(r,s)}_{x_\Lam}=\set{\on{diag}\big(\sig_1(\al)\dots\sig_r(\al),\underbrace{\dots\sig_{r+i}(\al),\bar{\sig}_{r+i}(\al)\dots}_{s}\big):\al\in\cO_{\Lam,+}^*}.$$
Hence, the assumption $\tilde{H}<\on{Ker}\pa{\chi_{ij}}$ implies that there are two distinct embeddings of $K$ $\tau,\eta$ (corresponding to the $i$'th and $j$'th diagonal entries) such that if $\al\in\cO_{\Lam,+}^*$ satisfies $\tilde{\psi}(\al)\in \tilde{H}$ then $\tau(\al)=\eta(\al)$. Let $K'=\set{\al\in K : \tau(\al)=\eta(\al)}$. Note that $K'$ is a proper subfield of $K$. Let $d_1=\deg\pa{K/K'}$ and $d_2=\deg\pa{K'/\bQ}.$ The different embeddings of $K$
partition into $d_2$ sets of equal size such that if $\tau'$ and $\eta'$ belong to the same partition set, then their restrictions to $K'$ coincide. The lemma now follows.
\end{proof}
The following two lemmas will be needed to complete the proofs of Theorems~\ref{theorem 1} and~\ref{theorem 2}. Recall that
a number field $K$ is a CM field if $K$ is of type $(0,s)$ and contains a totally real subfield of degree $s=d/2$ over $\bQ$.
\begin{lemma}\label{centralizer 1}
Let $x_\Lam\in X_d$ be a lattice coming from a number field $K$ of type $(r,s)$ with $s>0$. Let $H<T^{(r,s)}$ be the subgroup satisfying $\overline{A_{r,s}x_\Lam}=Hx_\Lam$. Then the following are equivalent:
\begin{enumerate}
	\item $H$ strictly contains $A_{r,s}$.
	\item $C_G(H)=T^{(r,s)}$.
	\item $K$ is not a CM field.
\end{enumerate}
\end{lemma}
\begin{proof}
We first show (3) $\Rightarrow$ (2). Suppose $K$ is not CM. $C_G(H)$ is strictly larger than $T^{(r,s)}$ if and only if $H$ does not contain any regular elements or equivalently, there exist $1\le i\ne j\le d$ such that $\tilde{H}<\on{Ker}(\chi_{ij})$. Lemma~\ref{equiblock} implies that in this case
$\tilde{H}_{x_\Lam}$ is an equiblock diagonal group with block size $d_1>1$. On the other hand $A_{r,s}=\tilde{A}_{r,s}<\tilde{H}$ by definition. As the orbit $Hx_\Lam$ is compact we conclude that $\tilde{H}/\tilde{H}_{x_\Lam}$ is compact too and in particular any element of $A_{r,s}$ can be brought to a
compact set when multiplied by an appropriate element of $\tilde{H}_{x_\Lam}$. It follows
that $r=0$, $d_1=2$, and that $\tilde{H}_{x_\Lam}\subset A_{0,s}$. This implies that $A_{0,s}x_\Lam$ is compact and that $H=A_{0,s}$ by definition. Lemma~\ref{equiblock} also implies that there is a subfield $K'\subset K$ of degree $d/2=s$ over $\bQ$ such that
$$H_{x_\Lam}\subset \psi\pa{\cO_{\Lam,+}^*\cap K'}.$$
As $H=A_{0,s}\simeq \bR^{s-1}$ and the quotient $H/H_{x_\Lam}$ is compact, we conclude that the rank of the group $\cO_{\Lam,+}^*\cap K'$ must be at least $s-1$. On the other hand, it follows from Dirichlet's unit theorem that the rank of this group is bounded above by $s-1$ and equality holds if and only if $K'$ is totally real. This implies that $K'$ is indeed totally real and that $K$ is a quadratic totally complex extension of it, i.e.\  that $K$ is a CM field which contradicts our assumption.\\
The implication (2) $\Rightarrow$ (1) is obvious since we assume $s>0$ and so the centralizer of $A_{r,s}$ strictly contains $T^{(r,s)}$. Finally, to see that (1) $\Rightarrow$ (3), assume that $K$ is a CM field. It follows that $\on{Stab}_{A_{0,s}}(x_\Lam)=\psi\pa{\cO_{\Lam,+}}$.
Dirichlet's unit theorem implies that $\on{rank}(\cO_{\Lam,+}^*)=s-1$, hence $A_{0,s}x_\Lam$ is compact and by definition $H=A_{0,s}$.
\end{proof}
\begin{lemma}\label{centralizer 2}
Let $x_\Lam\in X_d$  be a lattice coming from a number field $K$ of type $(r,s)$. Let $k$ be a number co-prime to $d$ such that $\set{a_k(t)}<A_{r,s}$, and let $H'<T^{(r,s)}$ be the subgroup satisfying $\overline{\set{a_k(t)x_\Lam}}_{t\in\bR}=H'x_\Lam$. Then $C_G(H')=T^{(r,s)}$.
\end{lemma}
\begin{proof}
$C_G(H')$ strictly contains $T^{(r,s)}$ if and only if $\tilde{H'}<\on{Ker}(\chi_{ij})$ for some $1\le i\ne j\le d$. If this is the case, then Lemma~\ref{equiblock} implies that $\tilde{H'}_{x_\Lam}$ is an equiblock group. As $\tilde{H'}/\tilde{H'}_{x_\Lam}$ is compact, it follows that large elements of the one parameter group $\set{a_k(t)}_{t\in\bR}<\tilde{H'}$ can be brought to a compact set if multiplied by the appropriate elements of the equiblock diagonal group $\tilde{H'}_{x_\Lam}$. This contradicts the assumption that $k$ is co-prime to $d$.
\end{proof}
\section{Proof of Theorems~\ref{theorem 1}, \ref{theorem 4}, and~\ref{theorem 2}.}\label{main proofs}
This section is organized as follows. We present below a strategy of proving that a point $x\in X_d$ is $A$-regular of periodic type, which is
the subject of Theorems~\ref{theorem 1},~\ref{theorem 2}. The discussion culminates in Lemma~\ref{main lemma} below, and then the theorems are derived. At the end of the section we deduce Theorem~\ref{theorem 4} from the
results appearing in~\cite{Shap1}.

In \cite{LW}, Barak Weiss and the first named author proved the following theorem:
\begin{theorem}\label{LW}
Let $d\ge 3$ and $x\in X_d$ be such that $\overline{Ax}$ contains a compact $A$-orbit, then $x$ is $A$-regular of periodic type.
\end{theorem}
The following is a consequence  of Theorem 2.8 in~\cite{PR}.
\begin{theorem}\label{PR}
Let $L$ be a reductive subgroup of $G$ containing $A$, and let $Lx$ be a periodic orbit of $L$ in $X_d$. Then $Lx$ contains a compact $A$-orbit.
\end{theorem}

Ratner's Measure Classification Theorem \cite[Theorem 1]{R-annals} gives a classification of measures in $X_d$ that are invariant under a one parameter unipotent subgroup of $G$. We shall require the following variant, proved in the next section:

\begin{theorem}\label{EKL}
Let $\mu$ be an $A$-invariant and ergodic probability measure on $X_d$ which is invariant under a one parameter unipotent subgroup of $G$. Then,
it is in fact an $L$-invariant probability measure supported on a single $L$-orbit
in $X_d$, for some reductive group $L$ containing $A$.
\end{theorem}
The above three theorems suggest a scheme of proving that a point $x\in X_d$ is $A$-regular of periodic type. Namely, one should prove
that $\overline{Ax}$ contains the support of an $A$-invariant probability measure which is invariant under a one parameter unipotent subgroup of $G$. The proofs of Theorems \ref{theorem 1} and \ref{theorem 2} follow this scheme.
To obtain an $A$-invariant measure in our arguments, we start with an initial probability measure $\nu$,
which is not $A$-invariant but is supported inside the orbit closure $\overline{Ax}$, we choose
a F\o lner sequence $F_n\subset A$, and define the averages
\begin{equation}\label{average 1}
\mu_n=\frac{1}{\av{F_n}}\int_{F_n}a_*\nu da,
\end{equation}
where $\av{F_n}$ denotes the Haar measure of $F_n$ in $A$.
Any weak$^*$ limit $\mu$ of the sequence $\mu_n$ will be an $A$-invariant measure on $X_d$. We face two problems
\begin{enumerate}
	\item One needs to prove that $\mu$ is a probability measure (i.e.\  there is no escape of mass).
	\item One needs to prove that $\mu$ is invariant under a one parameter unipotent subgroup of $G$.
\end{enumerate}
The fact which enables us to overcome the above problems is the nature of the initial probability measure $\nu$. We shall see in the course of
the arguments that $\nu$ is chosen to be an $H$-invariant probability measure supported on an orbit $Hy\subset \overline{Ax}$, for some
suitable choice of a point $y\in\overline{Ax}$ and a subgroup $H<G$ (having some additional properties). The tool which enables us to resolve problem (1) is the following theorem of Eskin, Mozes, and Shah (see \cite{EMS-Gafa}):
\begin{theorem}\label{EMS Gafa}
Let $H$ be a reductive subgroup of $G$ and let $\nu$ be an $H$-invariant measure supported on an orbit $Hy\subset X_d$. If the orbit $C_G(H)y$ of the centralizer of $H$ in $G$ is compact, then for any sequence $g_n\in G$, any weak$^*$ limit of $(g_n)_*\nu$ is a probability measure.
\end{theorem}
The following lemma is needed for the resolution of problem (2). It shows us how to choose the F\o lner sets in~\eqref{average 1} in order that $\mu$ will indeed be invariant under a unipotent one parameter subgroup of $G$. The proof is postponed to the next section.
\begin{lemma}\label{gaining unipotents}
Let $H<G$ be a closed connected subgroup
not contained in $A$ and let $\nu$ be an $H$-invariant probability measure. There exists an open cone $C$ in $A$ and a unipotent one parameter subgroup $u(t)$ in $G$, such that if the F\o lner sets $F_n$ are contained in $C$, then any weak$^*$ limit of $\mu_n$ from~\eqref{average 1} is $u(t)$-invariant.
\end{lemma}
We summarize the above discussion in the form of the following lemma:
\begin{lemma}\label{main lemma}
Let $d\ge 3$ and $x\in X_d$ be given. The following  implies that $x$ is $A$-regular of periodic type: There exists a closed connected reductive subgroup $H$ of $G$ and a point $y\in\overline{Ax}$ with the following properties:
\begin{enumerate}
	\item $H$ is not contained in $A$.
	\item $Hy\subset \overline{Ax}$ and $Hy$ supports an $H$-invariant probability measure.
	\item The orbit $C_G(H)y$ is compact.
\end{enumerate}
\end{lemma}
\begin{proof}[Concluding the proof of Theorem~\ref{theorem 1}]
If $s=0$, then $Ax_\Lam$ is compact by Lemma~\ref{compact orbits} and hence $x_\Lam$ is $A$-regular of periodic type. Assume that $s>0$. We wish to use Lemma~\ref{main lemma} with the following choices of $y$ and $H$. Let $y=x_\Lam$ and let $H<T^{(r,s)}$ be the closed subgroup defined by the equation $\overline{A_{r,s}x_\Lam}=Hx_\Lam$. As we assume that the number field is not CM, it follows from Lemma~\ref{centralizer 1} that $H$ is not contained in $A$. It is clear that $H$ is connected, reductive, and that the $H$-orbit $Hx_\Lam\subset\overline{Ax_\Lam}$ supports an $H$-invariant probability measure. Lemma~\ref{centralizer 1} implies that $C_G(H)=T^{(r,s)}$ and Lemma~\ref{compact orbits} implies that $C_G(H)x_\Lam$ is compact.
We see that the conditions of Lemma~\ref{main lemma} are satisfied and the theorem follows.
\end{proof}
\begin{proof}[Concluding the proof of Theorem \ref{theorem 2}]
Let $F=\overline{Apx_\Lam}$. Observe that $F$ contains $$a_k(t)pa_k(-t)a_k(t)x_\Lam.$$ Assume for example that $p\in P_k^-$. Note that as $t\to \infty$, the conjugation
$a_k(t)pa_k(-t)$ approaches a limit $p'\in G$, while $a_k(t)x_\Lam$ has as limit points, any point in $H'x_\Lam$, where $H'<T^{(r,s)}$ is defined as in Lemma~\ref{centralizer 2}. Denote $y=p'x_\Lam$ and $H=p'H'p'^{-1}$. We see that $H$ is connected, reductive, and $F$ contains the orbit $Hy$ which supports an $H$-invariant probability measure. Moreover, Lemmas~\ref{centralizer 2},~\ref{compact orbits} imply that $y$ has a compact orbit under the action of the centralizer
\begin{equation}\label{eq:centralizer}
C_G(H)=p'C_G(H')p'^{-1}=p'T^{(r,s)}p'^{-1}.
\end{equation}
The argument now splits into two possibilities. Assume $H$ is not contained in $A$. Then Lemma~\ref{main lemma} applies and the theorem is proved. Assume on the other hand that $H<A$. It follows from~\eqref{eq:centralizer} that $A=p'T^{(r,s)}p'^{-1}$, hence $s=0$, $A=T^{(r,s)}$, and $p'\in N_G(A)$.
We conclude that $F$ contains the compact orbit $Ay=p'Ax_\Lam$ (by Lemma~\ref{compact orbits}). Theorem~\ref{LW} applies and the theorem follows.
\end{proof}
Note that Theorems~\ref{LW},~\ref{PR},~\ref{EKL} imply together the following characterization of $A$-regular points of periodic type in $X_d$ ($d\ge 3$).
\begin{theorem}\label{theorem 6}
Let $d\ge 3$ and $x\in X_d$. Then $x$ is $A$-regular of periodic type if and only if $\overline{Ax}$ contains a compact $A$-orbit.
\end{theorem}
\begin{corollary}[Inheritance]\label{generalization of LW}
Let $d\ge 3$ and $x\in X_d$ be such that $\overline{Ax}$ contains a point $y$ which is $A$-regular of periodic type. Then, $x$ is $A$-regular of periodic type too.
\end{corollary}
We end this section by deducing Theorem~\ref{theorem 4} from the results in~\cite{Shap1}.
\begin{proof}[Proof of Theorem~\ref{theorem 4}]
In \cite[Theorem 4.5]{Shap1} it is stated that if $d\ge 3$ and $x\in X_d$ is such that $\overline{Ax}$ contains a compact $A$-orbit, then either $Ax$ is compact, or $x$ is GDP. It now follows from Theorem~\ref{theorem 6}, that our assumption that $x$ is $A$-regular of periodic type, implies that $\overline{Ax}$ contains a compact $A$-orbit, and the theorem follows.
\end{proof}
\section{Proofs of Lemma~\ref{gaining unipotents} and Theorem~\ref{EKL}}
\subsection{Preliminaries} In order to present the proofs of Lemma~\ref{gaining unipotents} and Theorem~\ref{EKL},
we need to introduce some terminology. We denote the Lie algebras of $G$ and $A$
by $\gog,\goa$ respectively.
$\goa$ consists of traceless diagonal matrices.
We have the \textit{root space decomposition}
\begin{align}\label{deco}
\gog=\goa\oplus\bigoplus_{i\ne j}\gog_{ij},
\end{align}
where the $\gog_{ij}$'s are the one dimensional common eigenspaces of $Ad_a$, $a\in A$. $\gog_{ij}$ is referred to as a \textit{root space}. Given a vector $X\in\gog$ we let $X=X_\goa+\sum_{i\ne j}X_{ij}$ denote its decomposition with respect to~\eqref{deco}. We denote by $\log$ the inverse
of the exponential map $\exp:\goa\to A$. Given a vector $v\in \goa$, the operator $\on{Ad}_{\exp(v)}$ has  $\gog_{ij}$ as a one dimensional eigenspace and it acts on it by multiplication by $e^{\lam_{ij}(v)}$, where $\lam_{ij}:\goa\to \bR$ is a linear functional called a \textit{root}. Hence, we have the following identity for $v\in\goa$ and $X\in \gog$
\begin{align}\label{formula}
\on{Ad}_{\exp(v)}(X)=X_\goa+\sum_{i\ne j}e^{\lam_{ij}(v)}X_{ij}.
\end{align}
The reader will easily convince himself that if $L<G$ is a closed connected subgroup with Lie algebra $\gol$, then $L$ is normalized by $A$ if and only if
\begin{align}\label{lie alg}
\gol=\pa{\goa\cap\gol}\oplus\bigoplus_{(i,j)\in I}\gog_{ij},
\end{align}
for some suitable choice of subset $I\subset\set{(i,j):1\le i\ne j\le d}$.

We order the roots in the following way: We say that $\lam_{ij}>\lam_{k\ell}$
if $j-i>k-\ell$ or if $j-i=k-\ell$ and $i<k$. In this way the ordering is linear. We say that a root $\lam_{ij}$ is \textit{positive}, if $i<j$. We fix some norm $\norm{\cdot}$ on $\gog$ and some metric $\on{d}(\cdot,\cdot)$ on $G$ (inducing the usual topology). In a metric space $Y$, we let  $B^Y_y(\rho)$ denote  the ball of radius $\rho$ around $y$ in $Y$. If $E,F\subset Y$, we let $\on{d}(E,F)$ denote the distance between the sets $E,F$.
Finally, for any diagonal matrix $a$ (not necessarily traceless), we denote by $p_\goa(a)$, its projection
to $\goa$, i.e.\
\begin{align}\label{projection on a}
p_\goa(a)=a-\diag{\frac{\on{Tr}(a)}{d},\dots,\frac{\on{Tr}(a)}{d}}.
\end{align}
\subsection{Proofs}
\begin{lemma}\label{first step}
Let $H<G$ be a closed connected subgroup not contained in $A$. Let $\goh$ denote its Lie algebra. Then, there exist an open cone
$\hat{C}<\goa$ and a nilpotent matrix $\textbf{n}\in\gog$, such that for any $\del>0$, there is a radius
$R>0$, such that if $v\in\hat{C}$ has norm $>R$, then $\on{d}\pa{\textbf{n},\on{Ad}_{\exp(v)}(\goh)}<\del$.
\end{lemma}
\begin{proof}[Deducing Lemma~\ref{gaining unipotents} from Lemma~\ref{first step}]
Let $C=\exp(\hat{C})$ and $F_n$ be a F\o lner sequence in $C$. Let $\mu_n=\frac{1}{\av{F_n}}\int_{F_n}a_*\nu da$ be as in~\eqref{average 1}, and let
$\mu$ be a weak$^*$ limit of the $\mu_n$'s. We shall prove that $\mu$ is invariant under the one parameter unipotent subgroup of $G$, given by $u(t)=\exp(t\textbf{n})$. To prove this, let $f$ be a continuous function with compact support on $X_d$. We need to show that the following equality holds for any $t\in\bR$,
\begin{align}\label{need to prove}
\int_{X_d} f(x)d\mu=\int_{X_d} f(u(t)x)d\mu.
\end{align}
We show this for $t=1$ for example, and denote $u=u(1)$. For convenience, we further assume that $\norm{f}_\infty\le 1$. Given $\eps>0$, we can find $n_0$ and $R$, sufficiently large and $\rho>0$ sufficiently small, so that the following four conditions  hold
\begin{enumerate}
\item For any $n>n_0$ we have $$\av{\int_{X_d} f(x)d\mu-\int_{X_d} f(x) d\mu_n}<\eps\textrm{ and }\av{\int_{X_d} f(ux)d\mu-\int_{X_d} f(ux) d\mu_n}<\eps.$$
\item For any $g\in G$ such that $\on{d}\pa{g,e}<\rho$, and for any $x\in X_d$, we have $$\av{f(gx)-f(x)}<\eps.$$
\item For any $a\in C$ with $\norm{\log(a)}>R$, there exist $h_a\in H$ and $g_a\in B^G_\rho(e)$ such that $$u=g_aah_aa^{-1}.$$
\item Finally, for any $n>n_0$ we have $$\frac{\av{F_n\cap \exp\pa{B_R^\goa}}}{\av{F_n}}<\eps.$$
\end{enumerate}
(1) follows from the definition of weak$^*$ convergence, (2) follows from the fact that $f$ is continuous and has compact support, (3) is a reformulation of the  conclusion of Lemma~\ref{first step}, and (4) just follows (if $n_0$ is sufficiently large) from the fact that $\av{F_n}\to \infty$ while $R$ is fixed.
To conclude the proof we have the following series of estimates which implies~\eqref{need to prove} when taking $\eps$ to zero. We marked the equalities and estimates below
to indicate which of the above properties is used in each passage. We use the symbol $\al\sim_\eps\be$ to denote that $\al$ and $\be$ are at most $\eps$ distance apart. The only unmarked equality is in the fifth line and the reason it holds is that for $a\in A$, the measure $a_*\nu$ is $aHa^{-1}$-invariant. For $n>n_0$ and $R$ as above, we have
\begin{align*}
\int_{X_d}f(ux)d\mu \;\;\stackrel{(1)}{{\sim}_\eps} &\int_{X_d}f(ux)d\mu_n\\
\stackrel{\on{def}}{=}&\;\;\frac{1}{\av{F_n}}\int_{F_n\cap \exp\pa{B_R^\goa}}\int_{X_d}f(ux)da_*\nu da +
\frac{1}{\av{F_n}}\int_{F_n\setminus \exp\pa{B_R^\goa}}\int_{X_d}f(ux)da_*\nu da\\
\stackrel{(4)}{\sim_\eps}&\;\;\frac{1}{\av{F_n}}\int_{F_n\setminus \exp\pa{B_R^\goa}}\int_{X_d}f(ux)da_*\nu da\\
\stackrel{(3)}{=}&\;\;\frac{1}{\av{F_n}}\int_{F_n\setminus \exp\pa{B_R^\goa}}\int_{X_d}f(g_aah_aa^{-1}x)da_*\nu da\\
=&\;\;\frac{1}{\av{F_n}}\int_{F_n\setminus \exp\pa{B_R}^\goa}\int_{X_d}f(g_ax)da_*\nu da\\
\stackrel{(2)+(3)}{\sim_\eps}&\;\;\frac{1}{\av{F_n}}\int_{F_n\setminus \exp\pa{B_R}^\goa}\int_{X_d}f(x)da_*\nu da\\
\stackrel{(4)}{\sim_\eps}&\;\;\frac{1}{\av{F_n}}\int_{F_n\cap \exp\pa{B_R}^\goa}\int_{X_d}f(x)da_*\nu da+\frac{1}{\av{F_n}}\int_{F_n\setminus \exp\pa{B_R}^\goa}\int_{X_d}f(x)da_*\nu da\\
\stackrel{\on{def}}{=}&\;\;\int_{X_d}f(x)d\mu_n \stackrel{(1)}{\sim_\eps} \int_{X_d}f(x)d\mu.
\end{align*}
 \end{proof}
\begin{proof}[Proof of Lemma~\ref{first step}]
As we assume that $H$ is not contained in $A$, we conclude that there is a root space (say of a positive root) $\gog_{ij}$ such that
\begin{equation}\label{root proj}
\pi_{ij}(\goh)=\gog_{ij}.
\end{equation}
Assume that $\lam_{i_0j_0}$ is the maximal positive root for which~\eqref{root proj} is satisfied with respect to the ordering of the roots described in the previous subsection. Let $X=X_\goa+\sum_{ij}X_{ij}\in \goh$ be such that $X_{i_0j_0}\ne 0$.
Let
\begin{align}\label{t}
v_0&=p_\goa\pa{\on{diag}\big(\underbrace{j_0,j_0-1,\dots,3,2}_{j_0-1},\underbrace{0,\dots,0}_{d-j_0+1}\big)}\in\goa.
\end{align}
The reader would easily verify that for any $(i,j)\ne(i_0,j_0)$ such that $X_{ij}\ne 0$
\begin{equation}\label{max}
\lam_{i_0j_0}\pa{v_0}\ge\lam_{ij}\pa{v_0}+1.
\end{equation}
It follows from continuity that we can choose an open cone, $\hat{C}\subset \goa$, containing the half line $\set{tv_0}_{t>0}$, such that for some $\eps>0$, the following holds: For any $v\in \hat{C}$ of norm $1$ and any $(i,j)\ne(i_0,j_0)$ such that $X_{ij}\ne 0$
\begin{equation}\label{max2}
\lam_{i_0j_0}\pa{v}\ge\lam_{ij}\pa{v}+\eps.
\end{equation}
Fix now $\del>0$ and let $R>0$ be given. Any vector in $\hat{C}$ of norm $>R$ is of the form $tv$ for $v\in \hat{C}$ of norm 1 and $t>R$. We now estimate the distance between $\on{Ad}_{\exp(tv)}(\goh)$ and the nilpotent matrix $X_{i_0j_0}\ne 0$. It follows
from~\eqref{formula},~\eqref{max2} that
\begin{align*}
\norm{\on{Ad}_{\exp(tv)}\pa{e^{-\lam_{i_0j_0}(v)t}X}-X_{i_0j_0}}&=
\norm{e^{-\lam_{i_0j_0}(v)t}\pa{X_\goa+\sum_{ij}e^{\lam_{ij}(v)t}X_{ij}}-X_{i_0j_0}}\\
&\le e^{-\lam_{i_0j_0}(v)t}\norm{X_\goa}+ \norm{\sum_{(i,j)\ne(i_0,j_0)}e^{\pa{\lam_{ij}(v)-\lam_{i_0j_0}(v)}t}X_{ij}}\\
&\le e^{-\lam_{i_0j_0}(v)t}\norm{X_\goa}+ \sum_{(i,j)\ne(i_0,j_0)}e^{-\eps t}\norm{X_{ij}}.
\end{align*}
 As the last expression goes to zero when $R\to \infty$ (recall that $t>R$), the lemma follows because $e^{-\lam_{i_0j_0}(v)t}X\in\goh$.
\end{proof}
We  now prove Theorem~\ref{EKL}
\begin{proof}[Proof of Theorem~\ref{EKL}]
Let $L$ denote the identity component of the closed subgroup
\begin{equation}\label{Lstab}
\on{Stab}_G(\mu)<G.
\end{equation}
Let $\gol$ denote its Lie algebra. It follows from~\eqref{lie alg}, that $\gol=\goa\oplus\bigoplus_{(i,j)\in I}\gog_{ij}$, where $I$ is some subset of $\set{(i,j):1\le i\ne j\le d}$. Our assumption that $\mu$ is invariant under a one parameter unipotent subgroup,
implies that $I$ is not trivial and in fact, $\mu$ is invariant under a one parameter unipotent group of the form $\set{u_{i_0j_0}(t)}=\exp(\gog_{i_0j_0})$. Choose an element $a\in A$ such that $\chi_{i_0j_0}(a)>1$. The expanding horospherical subgroup of $G$ with respect to $a$ is $$G_a^+=\exp\pa{\oplus_{\set{ij:\chi_{ij}(a)>1}}\gou_{ij}}.$$ It follows from~\cite[Theorem 7.6]{EL1} that $\on{h}_\mu(a)=\on{h}_\mu(a,G_a^+)$, where $\on{h}_\mu\pa{a,G_a^+}$ is the so called ``entropy contribution'' of $G_a^+$. Also, from~\cite[Corollary 9.10]{EL1} it follows that
\begin{equation}\label{entcont}
\on{h}_\mu(a,G_a^+)=\sum_{\set{ij:\chi_{ij}(a)>0}}\on{h}_\mu\pa{a,\set{u_{ij}(t)}}.
\end{equation}
From~\cite[Theorem 7.9]{EL1} we conclude that as $\mu$ is $\set{u_{i_0j_0}(t)}$-invariant, the summand $\on{h}_\mu\pa{a,\set{u_{i_0j_0}(t)}}$ in the
right hand side of~\eqref{entcont} equals $\log \av{\det Ad_a|_{\gou_{i_0j_0}}}>0$.  Hence we deduce that $\on{h}_\mu(a)>0$. We can now apply~\cite[Theorem 1.3]{EKL} to conclude that $\mu$ is the $L$-invariant probability measure supported on a periodic $L$-orbit. Note that although it is not stated explicitly in~\cite[Theorem 1.3]{EKL}
that $L$ is reductive, it is proved there that this is indeed the case (see also~\cite{LW} for a full classification of the possible groups $L$ which could arise in this way).
 \end{proof}

Appealing
to~\cite[Theorem 1.3]{EKL} is slightly artificial as the main difficulty in its proof is to use positivity of entropy to
deduce invariance under a unipotent. Here we start with a measure which is already invariant under a unipotent. Moreover, \cite[Theorem 1.3]{EKL} is only applicable for $d \geq 3$ (which is the case we are interested in), whereas Theorem~\ref{EKL} also holds for $d = 2$.

We sketch below an alternative argument, based directly on Ratner's Measure Classification Theorem \cite[Thm. 1]{R-annals} (a similar argument can be found in \cite{MT}).

\begin{proof}[Alternative proof of Theorem~\ref{EKL}]
Suppose $\mu$ is $A$-invariant and ergodic probability measure on $X _ d$ invariant under a one parameter unipotent subgroup $u _ t$. Since $\mu$ is also invariant under $a u _ t a^{-1}$ for any $a \in A$, and since the group of $g \in G$ preserving $\mu$ is closed, by going to the limit we may assume that $\{u _ t\}$ is normalized by $A$.

Let $\mu = \int _ \Xi \mu _ \xi \,d \rho (\xi)$ be the ergodic decomposition of $\mu$ with respect to the action of $u _ t$. Let $L _ \xi$ denotes the connected component of identity of $\operatorname{stab} _G \mu _ \xi$. By Ratner's Measure Classification Theorem, for $\rho$-almost every $\xi$ the measure $\mu _ \xi$ is supported on a single periodic $L _ \xi$-orbit. Since $A$ normalizes $\{u_t\}$, the group $A$ acts on the space of $\{u_t\}$-ergodic components $\Xi$, and moreover, by $A$-invariance of $\mu$, for any $a \in A$ and $\xi \in \Xi$ it holds that
\begin{equation*}
L _ {a \xi} = a L _ \xi a^{-1}
.\end{equation*}
By Poincare recurrence for the action of $A$ it is easy to deduce that $L _ \xi$ is a.s.\ normalized by $A$, and by ergodicity it follows that there is a connected group $L$ so that $L _ \xi = L$\ $\rho$-almost everywhere.

If $L$ were not reductive, one can find an element $a \in A$ so that $\det Ad(a)|_{\lie L}<1$, and since we have already shown that $\mu$-a.e.\ $x \in X _ d$ lies on a periodic $L$-orbit it follows that in this case $a ^ nx \to \infty$\ $\mu$-a.e.: in contradiction to Poincare recurrence. Moreover, if $Lx$ is periodic so is $[L,L]x$. Since $u_t \in [L,L]$, the natural measure
on $Lx$ will not be $u_t$ ergodic unless $L=[L,L]$ --- in contradiction to the construction of $L$ using the ergodic decomposition of $\mu$. Therefore $L$ is semisimple. Similarly $L x$ periodic for a semisimple $L$ implies that $H x$ is closed for $H=N _ G (L) ^ 0$ (this can be deduced e.g.\ from Lemma~\ref{closed orbits'} below).

Finally, $H = A L$ unless $L$ fixes some vector in $\bR ^ d$. But then since $L x$ is periodic, for any $\ell \in L$ we have that
$\left\{ (\ell - 1)v : v \in x \right\} $ is a nontrivial proper additive subgroup of the lattice $x$, which for an appropriate choice of inner product in $\bR ^ d$ is contained in the orthogonal complement to the subspace of $\bR ^ d$ fixed by $L$. It follows that $x$ intersects nontrivialy an $L$-invariant proper subspace of $\bR ^ d$, and since $L$ is normalized by $A$ one can find an element $a \in A$ contracting the subspace, hence again $a^nx \to \infty$: which cannot happen for a typical $x$.

To conclude we have shown that $\mu$ is $A$ and $L$-invariant and supported on a single closed orbit of $H=AL$, hence this orbit must be periodic and we are done.
\end{proof}

\section{Application to Diophantine approximations}\label{applications}
The proof of Theorem~\ref{theorem 5}, which is the subject of this section, is merely a simple application of Theorem~\ref{theorem 2}.
\begin{proof}[Proof of Theorem~\ref{theorem 5}]
Let $\Lam=\on{Span}_\bZ\set{1,\al_2,\dots,\al_d}$ be the lattice spanned by the $\bQ$-basis $\set{1,\al_2,\dots,\al_d}$, of the number field $K$. Let $\vphi$ be a geometric embedding of $K$ in $\bR^d$ (as in~\eqref{g.e}) and assume that the first embedding is chosen to be the identity. Let $x_\Lam\in X_d$ be the lattice
corresponding to $\Lam$. Let
$P_1^-=\set{\pa{
\begin{array}{ll}
b&0\\
\vec{c}&D
\end{array}}\in G}$ be as in~\eqref{P_k} (i.e.\ $b\in \bR,\vec{c}\in\bR^{d-1}$ and $D\in\GL_{d-1}(\bR)$), and $a_1(t)$ be as in~\eqref{a_k}. The reader will easily verify that there exists a nonzero constant $c\in\bR$ and $p\in P_1^-$
such that
\begin{equation}
\pa{
\begin{array}{ll}
1&v^t\\
0 & I_{d-1}
\end{array}
}=
cp\pa{\begin{array}{llll}
\vdots &\vdots& &\vdots\\
\vphi(1)&\vphi(\al_2)&\dots&\vphi(\al_d)\\
\vdots &\vdots & &\vdots
\end{array}},
\end{equation}
where the matrix on the the right in the above equation, has $\vphi(1),\vphi(\al_i)$ as its columns (here $c$ is just the inverse of the determinant of the lattice $\vphi(\Lam)$). It follows that $z_v=px_\Lam$. Theorem~\ref{theorem 2} now implies that $z_v$ is $A$-regular of periodic type (note that indeed the conditions of Theorem~\ref{theorem 2} are satisfied; i.e.\ 1 is co-prime to $d$ and if $K$ is of type $(r,s)$, then $r\ge 1$ and $\set{a_1(t)}<T^{(r,s)}$).
We now argue that $x_v$ is $A$-regular of periodic type too.

Denote by $x^*$ the dual lattice to a lattice $x\in X_d$ and the involution $g\mapsto (g^{-1})^t$ on $G$ by $g\mapsto g^*$. For any $g\in G$ and $x\in X_d$, $(gx)^*=g^*x^*$, and as $e_\Gamma$ is self-dual we have that $ge_\Gamma = g^*e_\Gamma$.
Hence $(z_{(-v)})^*=x_v$, and it follows that
 if $L<G$ is such that that $\overline{Az_{(-v)}}=Lz_{(-v)}$ is a finite volume orbit, then
 \begin{align}\label{duality}
 \overline{Ax_v}=\overline{Az_{(-v)}^*}=\pa{\overline{\pa{Az_{(-v)}}}}^*=\pa{Lz_{(-v)}}^*=L^*x_v.
 \end{align}
Hence $x_v$ is $A$-regular of periodic type. Note that in~\eqref{duality} we used the fact that $A^*=A$. In fact, it is not hard to see that $L^*=L$ too.
\end{proof}
\section{Examples of $A$-irregular points in $X_3$}\label{irregular}
In this section we shall prove Theorem~\ref{irregular points theorem}, indeed prove the somewhat more precise Theorem~\ref{theorem 3'}.
Recall that for a vector $v=(\al,\be)^t\in\bR^2$ we denote by $x_v,z_v$, the lattices in $X_3$ which are spanned by the columns of the matrices
\begin{equation}\label{another time}
h_v=\begin{pmatrix}
1&0&0\\
\al&1&0\\
\be&0&1
\end{pmatrix},\;
g_v=\begin{pmatrix}
1&\al&\be\\
0&1&0\\
0&0&1
\end{pmatrix},
\end{equation}
respectively.
 We first note that any statement about the $A$-orbit closure of the lattice $x_v$ immediately implies a corresponding statement for the lattice $z_{-v}$. This is because (as in~\eqref{duality}) $\overline{Az_{-v}}=\pa{\overline{Ax_v}}^*$.
Hence from now on we confine our discussion to lattices of the form $x_v$.
Before we turn to state Theorem~\ref{theorem 3'} we need to state some lemmas and introduce some notation. The following is well known.
\begin{lemma}\label{closed orbits'}
Let $\rho:\SL_d(\bR)\to \on{GL}(V)$ be a $\bQ$-representation, $v_0\in V$ a rational vector, and $H=\set{g\in\SL_d(\bR):\rho(g)v_0=v_0}$. Then the orbit $H\SL_d(\bZ)$ is closed in $\SL_d(\bR)/\SL_d(\bZ)$.
\end{lemma}
Let
\begin{equation}
H^{(1)}=\set{\pa{
\begin{array}{lll}
*&*&0\\
*&*&0\\
0&0&*
\end{array}
}\in G}\;;\;
H^{(2)}=\set{\pa{
\begin{array}{lll}
*&0&*\\
0&*&0\\
*&0&*
\end{array}
}\in G}.
\end{equation}
For each $i$,  the orbit of the group $H^{(i)}$ through the identity coset $e_\Ga\in X_3$ is closed by Lemma~\ref{closed orbits'}. For example, for $i=1$, one takes the appropriate exterior product of the adjoint representation of $G$ on its Lie algebra, and $v_0$ a rational
vector corresponding to the one parameter subgroup $\set{\diag{e^t,e^t,e^{-2t}}:t\in\bR}$ (as $H^{(1)}$ is the centralizer in $G$ of this one parameter subgroup, it is equal to the stabilizer of $v_0$). It now follows that for any matrix $p\in\SL_3(\bQ)$, the orbits,
$H^{(i)}pe_\Ga$, are closed in $X_3$ (this is done by considering the conjugations of $H^{(i)}$ by $p^{-1}$). For a positive integer $q$, let us consider the following closed orbits:
\begin{align}\label{M 1'}
M^{(1)}_q&=H^{(1)}\pa{\begin{array}{lll}
1&0&0\\
0&1&0\\
q^{-1}&0&1
\end{array}}e_\Ga\\ &\nonumber =\set{\pa{\begin{array}{lll}
a&b&0\\
c&d&0\\
\frac{1}{q\det B}&0&\frac{1}{\det B}
\end{array}}e_\Ga\in X_3:B = \pa{\begin{array}{ll} a&b\\c&d\end{array}}\in \on{GL}_2(\bR)};
\end{align}
\begin{align}\label{M 2'}
M^{(2)}_q&=H^{(2)}\pa{\begin{array}{lll}
1&0&1\\
q^{-1}&1&0\\
1&0&1
\end{array}}e_\Ga\\ &\nonumber =\set{\pa{\begin{array}{lll}
a&0&b\\
\frac{1}{q\det B}&\frac{1}{\det B}&0 \\
c&0&d
\end{array}}e_\Ga\in X_3:B = \pa{\begin{array}{ll} a&b\\c&d\end{array}}\in \on{GL}_2(\bR)}.
\end{align}
We will prove that the accumulation points of the orbit $Ax_v$ belong to $M^{(i)}_q$, for certain $q$'s, hence we wish to have a convenient characterization of the lattices composing $M^{(i)}_q$. This characterization is given by the following simple lemma.
\begin{lemma}\label{char of M}
A lattice $x\in X_3$ belongs to $M^{(1)}_q$ if and only if there exists $a\in A, g\in G$, and integers $\ell_1,\ell_2$, which generate $\bZ/q\bZ$ such that $x=age_\Ga$, and $g$ is of the following form
\begin{equation}\label{particular form 1}
g=\begin{pmatrix}
*&*&0\\
*&*&0\\
\frac{\ell_1}{q}&\frac{\ell_2}{q}&1
\end{pmatrix}.
\end{equation}
A lattice $x\in X_3$ belongs to $M^{(2)}_q$ if and only if there exists $a\in A, g\in G$, and integers $\ell_1,\ell_2$, which generate $\bZ/q\bZ$ such that $x=age_\Ga$, and $g$ is of the following form
\begin{equation}\label{particular form 2}
g=\begin{pmatrix}
*&0&*\\
\frac{\ell_1}{q}&1&\frac{\ell_2}{q}\\
*&0&*
\end{pmatrix}.
\end{equation}
\end{lemma}
\begin{remark}\label{generation}
It is a simple exercise to show that $\ell_1,\ell_2$ generate $\bZ/q\bZ$ if and only if there exists a matrix $\pa{k_{ij}}\in\SL_2(\bZ)$ such that $(\ell_1,\ell_2)=(1,0)\pa{k_{ij}} \on{mod}q$.
\end{remark}
\begin{proof}
We shall prove the lemma for $M^{(1)}_q$, leaving the corresponding statement for $M^{(2)}_q$ to the reader. Let $x\in M^{(1)}_q$. It follows from~\eqref{M 1'}, that up to the action of $A$, $x=ge_\Ga$, with $g$ in the form of~\eqref{particular form 1}, with $\ell_1=1,\ell_2=0$. For the other direction, let $x=age_\Ga$, where $a\in A$ and $g$ is in the form of~\eqref{particular form 1}.
By Remark~\ref{generation}, as $\ell_1,\ell_2$ generate $\bZ/q\bZ$, there exists a matrix $\pa{k_{ij}}\in\SL_2(\bZ)$, such that $(\ell_1,\ell_2)\pa{k_{ij}}=(1,0) \on{mod}(q)$, i.e.\  $k_{11}\ell_1+k_{21}\ell_2=n_1q+1$, and $k_{12}\ell_1+k_{22}\ell_2=n_2q$, for some integers $n_1,n_2$. Let $\ga\in \Ga$ be the matrix
$$\ga=\begin{pmatrix}
k_{11}&k_{12}&0\\
k_{21}&k_{22}&0\\
-n_1&-n_2&1
\end{pmatrix}.$$
Then a short calculation shows that the matrix $ag\ga$, which represents the lattice $x$, is of the form given in~\eqref{M 1'}, i.e.\ the lattice $x$ belongs to $M^{(1)}$ as desired.
\end{proof}
Finally, let $$a^{(1)}(t)=\diag{e^{-t},e^t,1},a^{(2)}(t)=\diag{e^{-t},1,e^t}.$$
\begin{theorem}\label{theorem 3'}
Let $v=(\al,\be)^t\in\bR^2$ be such that $\al,\be$ are irrational and $1,\al,\be$ linearly dependent over $\bQ$. Suppose
$\be=\frac{p_1}{q}\al+\frac{p_2}{q}$, $\al=\frac{p_1'}{q'}\be+\frac{p_2'}{q'}$, where these equations are written in reduced forms, i.e.\ $q$ (resp.\ $q'$) is a positive integer, and $p_1,p_2$ (resp.\ $p_1',p_2'$) generate $\bZ/q\bZ$ (resp.\ $\bZ/q'\bZ$). Then the following holds:
\begin{enumerate}
\item\label{T1} The orbit $Ax_v$ is disjoint from $M^{(1)}_q\cup M^{(2)}_{q'}$.
\item\label{T2} If $a_n\in A$ is a sequence such that the distance from $a_n$ to the two rays $\cup_{i=1,2}\set{a^{(i)}(t)}_{t>0}$, goes to $\infty$,
then the sequence $a_nx_v$ diverges (i.e.\ it has no converging subsequences in $X_3$).
\item\label{T3} Let $\Om_i=\set{\overline{\set{a^{(i)}(t)x_v}}_{t>0}\setminus\set{a^{(i)}(t)x_v}_{t>0}}$, and write $A\Om_i=\cup_{a\in A}a\Om_i$.
Then $$\overline{Ax_v}\setminus Ax_v=A\Om_1\cup A\Om_2\quad\textrm{and}\quad A\Om_i\ne\emptyset\quad\textrm{for } i=1,2.$$
\item\label{T4} Finally, $A\Om_1\subset M^{(1)}_q$, and $A\Om_2\subset M^{(2)}_{q'}.$
\end{enumerate}
\end{theorem}
\begin{proof}
We first argue why part~\eqref{T1} of the theorem follows from the fact that both $\al,\be$ are irrational. Working with~\eqref{M 1'}, we see that it suffice to show that there is no $\ga\in \Ga$ such that
\begin{equation}\label{cant rep}
\pa{\begin{array}{lll}
1&0&0\\
\al&1&0 \\
\be&0&1
\end{array}}\ga=
\pa{\begin{array}{lll}
a&b&0\\
c&d&0\\
\frac{1}{q\det B}&0&\frac{1}{\det B}
\end{array}}\;,\textrm{ where }B=\pa{\begin{array}{ll} a&b\\c&d\end{array}}\in \on{GL}_2(\bR).
\end{equation}
An analogue statement should be verified when working with~\eqref{M 2'}. In order to argue why there is no $\ga\in\Ga$ solving~\eqref{cant rep}, note first that the rightmost column of $\ga$ must be of the form $(0,0,\det B)^t$. This implies that $\det B$ is an invertible integer i.e.\ $\det B=\pm 1$; this follows because the determinant of $\ga$ equals $\det B$ times the determinant of the two by two upper left block of $\ga$, which is also an integer. It now follows that as $\be$ is irrational, the leftmost and middle columns of $\ga$, must be of the form $(0,*,*)^t$. Hence the first row of $\ga$ equals zero, a contradiction.

We now prove part~\eqref{T2} of the theorem. Let $a_n=\diag{e^{-t_n-s_n},e^{s_n},e^{t_n}}\in A$ be a diverging sequence (i.e.\ $\av{t_n}+\av{s_n}\to\infty$), such that $a_nx_v\to x,$ for some $x\in X_3$. Our goal is to show that the sequence
$\min\set{\av{t_n},\av{s_n}}$ is bounded. We will use the following fact about converging sequences of lattices;  a converging sequence in $X_3$ has a
positive lower bound on the lengths of the shortest nonzero vectors of its elements.
We first argue that both $t_n$ and $s_n$ are bounded from below.
This is because the lattice $x_v$ contains the standard basis vectors $e_2,e_3$, and if for instance $t_n$ is not bounded from below, then $a_nx_v$
contains the nonzero vector $a_ne_3$ which is arbitrarily short when $t_n$ is negative and arbitrarily large in absolute value. Hence, we can assume
that $t_n,s_n\ge 0$ (this is done by replacing the sequence $a_n$ by a constant multiple of it, $aa_n$, if necessary).
We now exclude the possibility of  $\min\set{t_n,s_n}$ being unbounded from above. We use Dirichlet's theorem which asserts that for any real number, $\theta\in\bR$, and any $T>0$, there exist $k,m\in\bZ$, with $0<\av{k}\le T$, such that
$$\av{k\theta+m}\le \frac{1}{kT}.$$
Using this theorem, we wish to produce vectors in $a_nx_v$ which will be arbitrarily short, once $\min\set{t_n,s_n}$ is arbitrarily large.
For a given $n$, choose
\begin{align}\label{T}
T=\Bigg\{\begin{array}{ll}
e^{t_n+\frac{s_n}{2}}&\textrm{if }t_n\ge s_n\\
e^{s_n+\frac{t_n}{2}}&\textrm{if } t_n<s_n
\end{array},
\end{align}
and apply Dirichlet's theorem for $\al$ and $T$ to conclude the existence of $k,m\in\bZ$ with $0<\av{k}<T$, and
$\av{k\al + m}<\frac{1}{kT}$. As $\be=\frac{p_1}{q}\al+\frac{p_2}{q}$, this implies that
$$\av{qk\be+(p_1m-p_2k)}=\av{p_1k\al+p_2k+p_1m-p_2k}\le\frac{p_1}{kT}.$$
We conclude that there exists a vector in the lattice $a_nx_v$, of the  form
\begin{align}\label{shorty}
a_nh_v\begin{pmatrix}
qk\\
qm\\
p_1m-p_2k
\end{pmatrix}=
\pa{\begin{array}{lll}
e^{-t_n-s_n}qk\\
e^{s_n}(qk\al+qm)\\
e^{t_n}(qk\be+(p_1m-p_2k))
\end{array}},
\end{align}
which has length $\le\max\set{e^{-t_n-s_n}qT,\frac{e^{s_n}q}{T},\frac{e^{t_n}p_1}{T}}\le \max\set{p_1,q} e^{-\min\set{s_n,t_n}/2}$. Where the last inequality follows from
~\eqref{T}. This quantity is of course arbitrarily small once $\min\set{s_n,t_n}$ is not bounded from above, which concludes the proof of the
part~\eqref{T2} of the theorem.

We now prove part~\eqref{T4} of the theorem postponing part~\eqref{T3} to the end. As $M^{(1)}_q$ (resp.\ $M^{(2)}_{q'}$) is $A$-invariant, it is enough to prove that $\Om_1\subset M^{(1)}_q$ (resp.\ $\Om_2\subset M^{(2)}_{q'}$). We shall
prove that if $t_i\nearrow\infty$ is such that $x=\lim_i a^{(1)}(t_i)x_v$,
then $x\in M^{(1)}_q$, leaving the analogue statement for $a^{(2)}(t)$ to the reader.
Let $$S= \set{\pa{
\begin{array}{lll}
*&*&0\\
*&*&0\\
*&*&1
\end{array}
}\in G}.$$
From Lemma~\ref{closed orbits'}, it follows that $Se_\Ga$ is a closed orbit, and as $\set{a^{(1)}(t)}<S$ we conclude that $x\in S\Ga$. This means that $x$ has a basis composing the columns of a matrix in $S$.
It now follows from Lemma~\ref{char of M} that part~\eqref{T4} of the theorem will follow once we show the following
\begin{description}
\item[Claim 1] Any vector $w\in x$ is of the form $(*,*,\frac{\ell}{q})^t$, for some $\ell\in\bZ$.
\item[Claim 2] There exist two  vectors $w_j=(*,*,\frac{\ell_j}{q})^t\in x, j=1,2$, where $\ell_1,\ell_2$ generate $\bZ/q\bZ.$
\end{description}

To prove claim~1, let $w_i\in a^{(1)}(t_i)x_v$ be a sequence of vectors converging to $w\in x$.
There are sequences of integers $k_i,m_i,n_i$ such that
\begin{equation}\label{the w_i's}
w_i
=a^{(1)}(t_i)h_v\begin{pmatrix}k_i\\ m_i \\n_i\end{pmatrix}
=\pa{\begin{array}{lll}
e^{-t_i}k_i\\
e^{t_i}(k_i\al+m_i)\\
k_i\be+n_i\end{array}}.
\end{equation}
As $t_i\nearrow \infty$, we conclude that $k_i\al+m_i\to 0$. In other words, the distance from $k_i\al$ to $\bZ$, which we denote by $\idist{k_i\al},$ approaches zero. This implies
that the distance from $k_i\be=\frac{p_1}{q}k_i\al+\frac{p_2k_i}{q}$
 to $\frac{1}{q}\bZ$ approaches zero as well. Hence the third coordinate of $w$, which is
the limit of $k_i\be+n_i$, belongs to $\frac{1}{q}\bZ$ as desired.
In fact, a closer look shows that
$$k_i\be=\frac{1}{q}\pa{p_1m_i+p_2k_i}+\frac{\idist{k_i\al}}{q}.$$
This shows that in order to derive claim 2 we need to find two  families of sequences of integers $k_i^{(j)},m_i^{(j)},n_i^{(j)}, j=1,2$,
such that the vectors $w_i^{(j)}$ as in~\eqref{the w_i's} converge (maybe after passing to a subsequence), and such that there exist some pair $0\le\ell_1,\ell_2\le q-1$, generating $\bZ/q\bZ$, such that for any $i$, $$ \ell_j=p_1m^{(j)}_i+p_2k^{(j)}_i \on{mod}q.$$ Note that the role of $n_i^{(j)}$ is not significant and these might be chosen so as to bring $k_i^{(j)}\be$ to the unit interval.

To motivate the arguments we note the following.
There is a natural projection from
the periodic orbit $Se_\Ga$ (in which our discussion takes place) to the space of two dimensional unimodular lattices $X_2$. This projection is defined by the following formula (as we will now mix dimensions 2 and 3, we will denote hereafter $\SL_d(\bR), \SL_d(\bZ)$ by $G_d, \Ga_d$ respectively):
$$\begin{pmatrix}
a&b&0\\
c&d&0\\
\theta_1&\theta_2&1
\end{pmatrix}e_{\Ga_3} \mapsto
\begin{pmatrix}
a&b\\
c&d
\end{pmatrix}e_{\Ga_2}.$$
We denote this projection by $\pi$. It follows that in our notation $\pi(x_v)=x_\al=\begin{pmatrix}1&0\\ \al&1\end{pmatrix}e_{\Ga_2}$. If we denote $a(t)=\diag{e^{-t},e^{t}}\in G_2$, then $$\pi\pa{a^{(1)}(t_i)x_v}=a(t_i)x_\al,$$
and from the continuity of $\pi$ we deduce that $a(t_i)x_\al$ converges to $\pi(x)$.

Let $(k^{(j)}_i,m^{(j)}_i)^t\in\bZ^2, j=1,2,$ be chosen so that the vectors
$\hat{w}_i^{(1)},\hat{w}_i^{(2)}\in a(t_i)x_\al$, given by
\begin{equation}
\hat{w}_i^{(j)}=a(t_i)\begin{pmatrix}1&0\\ \al&1\end{pmatrix}\begin{pmatrix}k_i^{(j)}\\ m_i^{(j)}\end{pmatrix},
\end{equation}
are such that $\hat{w}_i^{(1)}$ is the shortest vector and $\hat{w}_i^{(2)}$ is the second shortest vector (not co-linear with $w_i^{(1)}$) in the lattice $a(t_i)x_\al$. As the first and second shortest vectors in a two dimensional lattice always form a basis, it follows that the matrix
$\ga=\begin{pmatrix} k_i^{(1)}&k_i^{(2)}\\m_i^{(1)}&m_i^{(2)}\end{pmatrix}$ has determinant $\pm 1$. Also, as $a(t_i)x_\al$ converges, the lengths of the vectors $\hat{w}^{(1)}_i$ are bounded from below, which in turn implies that the lengths of the vectors $\hat{w}^{(2)}$, are bounded from above (this follows from the fact that the covolume of the lattice $a(t_i)x_\al$ is equal to 1). It follows that, possibly after passing to a subsequence, we may assume that both sequences $\hat{w}_i^{(j)}$ converge. By passing to another subsequence if necessary, we may assume that the residue classes $p_1m_i^{(j)}+p_2k_i^{(j)}\on{mod}q, j=1,2$ are fixed and equal $\ell_j, j=1,2.$ It follows from
Remark~\ref{generation} that the equality $(\ell_1,\ell_2)=(p_2,p_1)\ga\on{mod}q$ forces $\ell_1,\ell_2$ to generate $\bZ/q\bZ$. We now choose $n_i^{(j)}$ in the manner specified above to define $w_i^{(j)}$ as in~\eqref{the w_i's} and note that the third coordinate of the $w_i^{(j)}$'s automatically
converges to $\frac{\ell_j}{q}$, hence the sequences $w_i^{(j)}$ converge to some limits $w_j=(*,*,\frac{\ell_j}{q})^t\in x$. The proof of part~\eqref{T4} follows.

The argument of part~\eqref{T3} of the theorem is divided into two parts. We first establish the equality $\overline{Ax_v}\setminus Ax_v =  A\Om_1\cup A\Om_2$, postponing the proof that the limit points form a nonempty set to the end.
It is clear from from parts~\eqref{T1},~\eqref{T4}, and the definition of $\Om_i$ that
$\overline{Ax_v}\setminus Ax_v\supset A\Om_1\cup A\Om_2$. To establish the opposite inclusion
let $x\in\overline{Ax_v}\setminus Ax_v$ be given. Then, there is a divergent sequence $a_n\in A$, such that $a_nx_v\to x$. From the second part of the theorem it follows that after passing to a subsequence, there is a sequence $t_n\to \infty$, such that $a_n=a_n'a^{(i)}(t_n)$, for $i=1$ or $i=2$, and $a_n'\to a\in A$. We conclude
that $a^{(i)}(t_n)x_v$ must converge to some point in $\Om_i$ (namely to $a^{-1}x$), and that $x=a\lim a^{(i)}(t_n)x_v\in A\Om_i$.

Finally, in order to finish the proof of the theorem we need to argue why
$\Om_j\ne\emptyset.$ Note that from parts~\eqref{T1} and~\eqref{T4} it follows that $$\Om_j=\set{x\in X_3: x=\lim a^{(j)}(t_i)x_v\textrm{ for some }t_i\nearrow\infty}.$$
Hence we only need to argue why the trajectories $\set{a^{(j)}(t)x_v}_{t\ge 0}$ are not divergent. We argue this for $j=1$ for example. Using the notation of the proof of part~\eqref{T4}, it is not hard to see that the projection $\pi:Se_{\Ga_3}\to X_2$ is proper. Hence, if the trajectory $\set{a^{(1)}(t)x_v}_{t\ge 0}$ is divergent, then so is its projection under $\pi$ in $X_2$, i.e.\ the trajectory $\set{a(t)x_\al}_{t\ge 0}$. It is well known that a geodesic ray in the upper half plane projects to a divergent geodesic ray in the modular surface if and only if its end point is rational. Since $\al$ is the endpoint here and it is irrational, a contradiction emerges. A slightly different argument showing that the orbit $Ax_v$ is not closed could be derived from Margulis classification of divergent $A$-orbits given in the appendix to~\cite{TW}.
\end{proof}

\begin{acknowledgments}
The results appearing in this paper are  part of the second named author's Ph.D thesis conducted at the Hebrew University of Jerusalem under the supervision of Hillel Furstenberg and Barak Weiss. We wish to thank them both as well as Hee Oh and Nimish Shah for valuable discussions, remarks and suggestions. We thank George Tomanov for communicating to us recent work on related topics and for helpful discussions. E.L. has had the good fortune of discussing these types of questions many times with Dan Rudolph; we regret that we no longer can enjoy his wisdom and kindness, and dedicate this paper to his memory.
\end{acknowledgments}


\begin{thebibliography}{Rat91b}

\bibitem[CSD55]{CaSD}
J.~W.~S. Cassels and H.~P.~F. Swinnerton-Dyer, \emph{On the product of three
  homogeneous linear forms and the indefinite ternary quadratic forms}, Philos.
  Trans. Roy. Soc. London. Ser. A. \textbf{248} (1955), 73--96. \MR{MR0070653
  (17,14f)}

\bibitem[Dav51]{D}
H.~Davenport, \emph{Indefinite binary quadratic forms, and {E}uclid's algorithm
  in real quadratic fields}, Proc. London Math. Soc. (2) \textbf{53} (1951),
  65--82. \MR{MR0041883 (13,15f)}

\bibitem[EKL06]{EKL}
Manfred Einsiedler, Anatole Katok, and Elon Lindenstrauss, \emph{Invariant
  measures and the set of exceptions to {L}ittlewood's conjecture}, Ann. of
  Math. (2) \textbf{164} (2006), no.~2, 513--560. \MR{MR2247967 (2007j:22032)}

\bibitem[EL10]{EL1}
M.~Einsiedler and E.~Lindenstrauss, \emph{Diagonal actions on locally
  homogeneous spaces}, Homogeneous flows, moduli spaces and arithmetic, Clay
  Math. Proc., vol.~10, Amer. Math. Soc., 2010, pp.~155--241.

\bibitem[EMS97]{EMS-Gafa}
A.~Eskin, S.~Mozes, and N.~Shah, \emph{Non-divergence of translates of certain
  algebraic measures}, Geom. Funct. Anal. \textbf{7} (1997), no.~1, 48--80.
  \MR{MR1437473 (98d:22006)}

\bibitem[LW01]{LW}
Elon Lindenstrauss and Barak Weiss, \emph{On sets invariant under the action of
  the diagonal group}, Ergodic Theory Dynam. Systems \textbf{21} (2001), no.~5,
  1481--1500. \MR{MR1855843 (2002j:22009)}

\bibitem[Mar89]{M-Oppenheim}
G.~A. Margulis, \emph{Discrete subgroups and ergodic theory}, Number theory,
  trace formulas and discrete groups ({O}slo, 1987), Academic Press, Boston,
  MA, 1989, pp.~377--398. \MR{MR993328 (90k:22013a)}

\bibitem[Mar97]{Margulis-Fields-medal}
\bysame, \emph{Oppenheim conjecture}, Fields {M}edallists' lectures, World Sci.
  Ser. 20th Century Math., vol.~5, World Sci. Publ., River Edge, NJ, 1997,
  pp.~272--327. \MR{MR1622909 (99e:11046)}

\bibitem[Mar00]{M-problems}
Gregory Margulis, \emph{Problems and conjectures in rigidity theory},
  Mathematics: frontiers and perspectives, Amer. Math. Soc., Providence, RI,
  2000, pp.~161--174. \MR{MR1754775 (2001d:22008)}

\bibitem[Mau]{Mau}
Francois Maucourant, \emph{A non-homogeneous orbit closure of a diagonal
  subgroup}, To appear in Ann' of Math'.

\bibitem[MT96]{MT}
G.~A. Margulis and G.~M. Tomanov, \emph{Measure rigidity for almost linear
  groups and its applications}, J. Anal. Math. \textbf{69} (1996), 25--54.
  \MR{MR1428093 (98i:22016)}

\bibitem[PR72]{PR}
Gopal Prasad and M.~S. Raghunathan, \emph{Cartan subgroups and lattices in
  semi-simple groups}, Ann. of Math. (2) \textbf{96} (1972), 296--317.
  \MR{MR0302822 (46 \#1965)}

\bibitem[Rat91a]{R-annals}
Marina Ratner, \emph{On {R}aghunathan's measure conjecture}, Ann. of Math. (2)
  \textbf{134} (1991), no.~3, 545--607. \MR{MR1135878 (93a:22009)}

\bibitem[Rat91b]{R-Duke}
\bysame, \emph{Raghunathan's topological conjecture and distributions of
  unipotent flows}, Duke Math. J. \textbf{63} (1991), no.~1, 235--280.
  \MR{MR1106945 (93f:22012)}

\bibitem[Sha]{Shap1}
Uri Shapira, \emph{A solution to a problem of cassels and diophantine
  properties of cubic numbers}, To appear in Ann' of Math', available on arXiv
  at http://arxiv.org/abs/0810.4289v2.

\bibitem[TW03]{TW}
George Tomanov and Barak Weiss, \emph{Closed orbits for actions of maximal tori
  on homogeneous spaces}, Duke Math. J. \textbf{119} (2003), no.~2, 367--392.
  \MR{MR1997950 (2004g:22006)}

\end{thebibliography}

\def\cprime{$'$}
\providecommand{\bysame}{\leavevmode\hbox to3em{\hrulefill}\thinspace}
\providecommand{\MR}{\relax\ifhmode\unskip\space\fi MR }
\providecommand{\MRhref}[2]{%
  \href{http://www.ams.org/mathscinet-getitem?mr=#1}{#2}
}
\providecommand{\href}[2]{#2}

\end{document}